\documentclass{amsart}
\usepackage{amsthm,amscd,amssymb,verbatim,epsf,amsmath,amsfonts,mathrsfs,graphicx,dirtytalk,pdfpages,mdframed,tikz-cd,xy}
\usepackage[colorlinks=true,linkcolor=blue,citecolor=blue]{hyperref}

\theoremstyle{plain}
\newtheorem{Thm}{Theorem}[section]
\newtheorem{Cor}[Thm]{Corollary}

\newtheorem{Prop}[Thm]{Proposition}
\theoremstyle{definition}

\newtheorem{Ex}[Thm]{Example}

\theoremstyle{remark}

\title[Fluid Flows with Straight Streamlines]{A Uniqueness Theorem for Incompressible Fluid Flows with Straight Streamlines}
\author{Brendan Guilfoyle}
\address{Brendan Guilfoyle\\
          School of STEM \\
          Munster Technological University\\
          Tralee  \\
          Co. Kerry \\
          Ireland.}
\email{brendan.guilfoyle@mtu.ie}

\date{\today}

\begin{document}
\begin{abstract}
    It is proven that the only incompressible Euler fluid flows with fixed straight streamlines are those generated by the normal lines to a round sphere, a circular cylinder or a flat plane, the fluid flow being that of a point source, a line source or a plane source at infinity, respectively.
    
    The proof uses the local differential geometry of oriented line congruences to integrate the Euler equations explicitly. 
\end{abstract}
\maketitle

In Newtonian gravitational theory, static fields whose lines of force are straight lines must be generated by the normal lines to either a round sphere, a circular cylinder or a flat plane. 

A similar situation holds in general relativity, where both static vaccuum and Weyl-type electrostatic gravitational fields with geodesic lines of force are generated by spheres, cylinders and planes, although new non-Newtonian solutions exist in both cases \cite{Das} \cite{g99} \cite{g00}.

Incompressible Euler flows model the non-relativistic hydrodynamics of fluids with no internal friction. They consist of a time-varying vector field $V$ on ${\mathbb R}^3$ called the {\em fluid velocity}, together with a time-varying function $p:{\mathbb R}^3\rightarrow {\mathbb R}$ called the {\it pressure}. The {\it streamlines} of the fluid are the integral curves of the fluid velocity, which in general evolve in time.

The Euler equations for an incompressible fluid in ${\mathbb R}^3$ are
\begin{equation}\label{e:euler1}
    \frac{\partial}{\partial t}V+\nabla_VV=-\nabla p
\end{equation}
\begin{equation}\label{e:euler2}
    \nabla\cdot V=0,
\end{equation}
where $\nabla$ is the Levi-Civita connection of the flat Euclidean metric. The second equation, the incompressibility condition, states that the fluid velocity vector is divergence-free. For foundational work on fluid flows see for example \cite{batch67}. For a modern overview of the Euler equations from a variety of perspectives see \cite{Constantin} \cite{gibbon08} and references therein.

The purpose of this paper is to prove:

\vspace{0.1in}
\noindent{\bf Main Theorem}:

{\it The only incompressible Euler fluid flows with fixed straight streamlines are the solutions generated by the normals to either a round sphere, a circular cylinder or a flat plane, the fluid flow being that of a point source, a line source or a plane source at infinity, respectively.
}

\vspace{0.1in}

This is proven as follows. The condition that the streamlines are straight implies that there exists a 2-parameter family of oriented lines (an {\it oriented line congruence}) to which the velocity vector is everywhere tangent. Using a special coordinate system fitted to the line congruence the Euler equations are explicitly integrated to yield the three solutions. 

 The trichotomy arises from the possible rank of the map that takes an oriented line to its direction, restricted to the line congruence of the fluid flow. This rank can either be two, one or zero, leading ultimately to the normals of a round sphere, a circular cylinder or a flat plane, respectively. 
 
The next section introduces some background on the local geometry of oriented line congruences. Further details can be found in \cite{gk08b}.

The three fluid solutions are discussed in Examples \ref{ex:sph}, \ref{ex:cyl} and  \ref{ex:pln} of Section \ref{s:3sols}. In Section \ref{s:inc} it is shown how every line congruence has a family of divergence-free vector fields tangent to it. These vector fields have singularities at the focal set of the line congruence. 

Having solved equation (\ref{e:euler2}) we turn to solving equations (\ref{e:euler1}). At a point in the fluid flow, the rank of the associated oriented line congruence is two, one or zero. Section \ref{s:rk2} contains the proof that rank two straight streamlines must be the normals to a round sphere. This is first proven for steady and then for non-steady flows.

Section \ref{s:rk1} contains the proof that rank one straight streamlines must be the normals to a circular cylinder. This is again first proven for steady and then for non-steady flows. The final section contains the proof of the rank zero case.

The Main Theorem is also likely to hold for the full incompressible Navier-Stokes equations \cite{constantinfioa}, as the introduction of dissipative effects would seem to make straight streamlines even less likely. 

On the other hand, given the non-Newtonian examples alluded to earlier, the relativistic Euler equations \cite{rezolla} may well admit solutions with geodesic streamlines that are not of the three above classes.

\vspace{0.1in}

\section{Oriented Line Congruences}\label{s:lc}

An {\em oriented line congruence} is a 2-parameter family of oriented lines in ${\mathbb R}^3$, or, equivalently, a surface $\Sigma$ in the space ${\mathbb L}({\mathbb R}^3)$ of all oriented lines.  The 4-manifold ${\mathbb L}({\mathbb R}^3)$ can be identified with the total space of the tangent bundle to the 2-sphere, $TS^2$. Thus it has a natural bundle structure $\pi:{\mathbb L}({\mathbb R}^3)\rightarrow S^2$, taking an oriented line to its direction. For semantic reasons, on occasion we drop the word oriented.

A oriented line congruence $\Sigma$ is {\em graphical} if the projection $\pi$ restricted to $\Sigma$ has rank two or, equivalently, it arises as the graph of a local section of this bundle. 

Taking the complex coordinate $\xi$ on $S^2$ given by stereographic projection from the North pole, a rank two oriented line congruence is given local by a map $\xi\mapsto(\xi,\eta=F(\xi,\bar{\xi}))$, where $\eta$ is a complex fibre coordinate and $F$ a complex function \cite{gk08b}. Here  $(\xi,\eta)\in{\mathbb C}^2$ are local coordinates on ${\mathbb L}({\mathbb R}^3)=TS^2$ minus the fibre over the South pole.

For such an oriented line congruence, define the shear $\sigma$, divergence $\theta$ and twist $\lambda$ by
\begin{equation}\label{e:spinco_rk2}
\sigma=-\frac{\partial \bar{F}}{\partial \xi}, \qquad \qquad \rho=\theta+i\lambda=(1+\xi\bar{\xi})^2\frac{\partial}{\partial \xi}\left(\frac{F}{(1+\xi\bar{\xi})^2}\right).
\end{equation}

On the other hand, a rank one oriented line congruence can be parameterized by $(u,v)\in{\mathbb R}^2$ via $(u,v)\mapsto(\xi(u),\eta(u,v))$. For such line congruences, define the real quantity
\begin{equation}\label{e:beta}
\beta=\frac{(\partial_v\eta\partial_u\bar{\eta}-\partial_u\eta\partial_v \bar{\eta})(1+\xi\bar{\xi})-2(\bar{\eta}\xi\partial_u\bar{\xi}\partial_v\eta-\eta\bar{\xi}\partial_u{\xi}\partial_v\bar{\eta})}{(\partial_v\eta\partial_u\bar{\xi}-\partial_v\bar{\eta}\partial_u\xi))(1+\xi\bar{\xi})}.
\end{equation}
A rank zero oriented line congruence consists of all oriented lines with a given direction $\xi_0\in S^2$, and are thus normal to a flat plane in ${\mathbb R}^3$. In terms of local coordinates, it can be given by
$(u,v)\mapsto(\xi_0,\eta=u+iv)$.

To connect with Euclidean 3-space with flat coordinates $(x^1,x^2,x^3)$, use the map  $\Phi:{\mathbb L}({\mathbb R}^3)\times{\mathbb R}\rightarrow{\mathbb R}^3$ that takes an oriented line $(\xi,\eta)$ and a number $r$, to the point $\Phi((\xi,\eta),r)\in{\mathbb R}^3$ which lies on the oriented line at an oriented distance $r$ from the point closest to the origin:
\begin{equation}\label{e:Phi}
z=x^1+ix^2=\frac{2(\eta-\bar{\xi}^2\bar{\eta})}{(1+\xi\bar{\xi})^2}+\frac{2\xi}{1+\xi\bar{\xi}}r,
\qquad\qquad
x^3=-\frac{2(\xi\bar{\eta}+\bar{\xi}\eta)}{(1+\xi\bar{\xi})^2}+\frac{1-\xi\bar{\xi}}{1+\xi\bar{\xi}}r.
\end{equation}
The distance from a point $(z,\bar{z},x^3)$ on the oriented line with direction $\xi$ to the point on the line closest to the origin is
\begin{equation}\label{e:r}
r=\frac{z\bar{\xi}+\bar{z}\xi+x^3(1-\xi\bar{\xi})}{1+\xi\bar{\xi}}.
\end{equation}

Given an oriented line congruence $\Sigma$, we can locally parameterize ${\mathbb R}^3$ by $U\times{\mathbb R}$, for $U\subset \Sigma$, using the parameter $r$ along each line. For rank two oriented line congruences discussed above, the coordinates are $(\xi,\bar{\xi},r)$, while for rank one and zero congruences they are $(u,v,r)$.

For these coordinate systems we must stay away from the focal set of the line congruence (if it has one) - which is at most two points on each line \cite{gk08b}. 

Oriented line congruences without focal points foliate ${\mathbb R}^3$ \cite{salvai09} and, aside from the parallel line case, are twisting everywhere. We will show that only the parallel case arises as straight streamlines of incompressible Euler fluid flows without focal points. For the other solutions, the focal points are singularities of the velocity vector and can be interpreted as fluid sources or sinks.

\vspace{0.1in}

\section{The Three Solutions}\label{s:3sols}

Consider a flow with straight streamlines. That is, at each point the fluid velocity $V$ is tangent to an oriented line in a 2-parameter family of oriented lines. An Euler flow is said to be {\em steady} if the fluid velocity and pressure are independent of time $t$. 

The following examples are the canonical solutions of the Euler equations with straight streamlines.  

\vspace{0.1in}
\begin{Ex}\label{ex:sph}
The set of normals to a round sphere generates a solution of the Euler equations with the fluid velocity vector and pressure given by
\[
    V=\frac{H(t)}{r^2}\frac{\partial}{\partial r}  \qquad\qquad  p=p_0-\frac{H^2}{2r^4}+\frac{\dot{H}}{r},
\]
for constant $p_0$, where $r$ is the distance to the centre of the sphere and $H$ a free function of time. The solution is steady if $H$ is constant.

The centre of the sphere is a singularity for the velocity and pressure, and the flow can be interpreted as a point source or sink. As $r\rightarrow\infty$ the fluid velocity goes to zero while the pressure becomes constant.

The associated rank two oriented line congruence is $\eta={\textstyle{\frac{1}{2}}}(z_0-2t_0\xi-\bar{z}_0\xi^2)$ where $(z_0,t_0)\in {\mathbb C}\times{\mathbb R}={\mathbb R}^3$ is the centre of the sphere. By direct computation using equations (\ref{e:spinco_rk2}) one finds that $\sigma=\rho=0$.

\end{Ex}
\vspace{0.1in}
\begin{Ex}\label{ex:cyl}
The set of normals to a circular cylinder generates a solution of the Euler equations with fluid velocity vector and pressure given by
\[
V=\frac{H(t)}{r}\frac{\partial}{\partial r}  \qquad\qquad  p=p_0-\frac{H^2}{2r^2}-\dot{H}\ln |r|,
\]
for constants $p_0$, where $r$ is the distance to the axis of symmetry of the cylinder and $H$ a free function of time. The flow is steady for $H$ constant.

The axis of symmetry is a singularity for the velocity and pressure, and the flow can be interpreted as a line source or sink. As $r\rightarrow\infty$ the fluid velocity goes to zero but only in the steady case is the pressure bounded.

After a suitable rotation and translation, the singularity can be lined up with the $x^2$ axis in ${\mathbb R}^3$, and the associated rank one oriented line congruence is $\xi=u,\eta=iv$. The function $\beta$ given by equation (\ref{e:beta}) then turns out to be zero.

\end{Ex}
\vspace{0.1in}
\begin{Ex}\label{ex:pln}
The set of normals to a flat plane generates a solution of the Euler equations and the fluid velocity vector and pressure are
\[
V=\left(H(t)+K(u,v))\right)\frac{\partial}{\partial r}  \qquad\qquad  p=p_0-\dot{H}r,
\]
for constant $p_0$, where $r$ is the distance to the plane and $u,v$ are parameters on the plane. The steady solution is finite, indeed constant, everywhere, while the pressure of the non-steady solution grows linearly at infinity.  This can be interpreted as a plane source at infinity.

Choosing the plane to be the $x^1x^2-$plane, the associated rank zero oriented congruence is simply $\xi=0,\eta=(u+iv)/2$.

\end{Ex}
\vspace{0.1in}

\section{Incompressibility}\label{s:inc}

The divergence-free condition (\ref{e:euler2}) models incompressible fluid flow. Every line congruence admits a family of tangent vectors which are divergence-free - one simply scales the unit tangent vector at each point by a suitable factor.

\vspace{0.1in}
\begin{Prop}
Let $\Sigma$ be a rank two line congruence given by $\eta=F(\xi,\bar{\xi})$. Then the vector field
\begin{equation}\label{e:rk2dv_free}
V=\frac{H(\xi,\bar{\xi})}{(r+\theta)^2+\lambda^2-|\sigma|^2}\frac{\partial}{\partial r},
\end{equation}
is divergence free for any real function $H$, where $\theta,\lambda$ and $\sigma$ are given by equations (\ref{e:spinco_rk2}).
\end{Prop}
\begin{proof}
Change from Euclidean coordinates $(z=x^1+ix^2,x^3)$ to congruence coordinates $(\xi,\bar{\xi},r)$ via equations (\ref{e:Phi}) with $\eta=F(\xi,\bar{\xi})$.
Pull back the flat metric in congruence coordinates and compute the divergence
\[
\nabla\cdot V=\nabla_k V^k=\frac{\partial V^k}{\partial x^k}+\Gamma_{kl}^kV^l,
\]
 with $ V^\xi=0$.  The resulting divergence-free condition can be integrated to yield the stated result. 
\end{proof}
\vspace{0.1in}

Note that this vector field blows up at the points where
\[
(r+\theta)^2+\lambda^2-|\sigma|^2=0,
\]
which is exactly the focal set of the line congruence \cite{gk08b}. 

\vspace{0.1in}
\begin{Prop}
Let $\Sigma$ be a rank one line congruence. Then the vector field
\begin{equation}\label{e:rk1dv_free}
V=\frac{H(u,v)}{r+\beta}\frac{\partial}{\partial r},
\end{equation}
is divergence free for any function $H$, where $\beta$ is given by equation (\ref{e:beta}).
\end{Prop}
\begin{proof}
from Euclidean coordinates $(z=x^1+ix^2,x^3)$ to congruence coordinates $(u,v,r)$ via equations (\ref{e:Phi}) with $(\xi(u),\eta(u,v))$.
Pull back the flat metric into congruence coordinates and compute the divergence with $ V^u=V^v=0 $.

The divergence-free condition can then be integrated for $V^r$.
\end{proof}

Finally, for rank zero oriented line congruences:

\vspace{0.1in}
\begin{Prop}
Let $\Sigma$ be a rank zero oriented line congruence - that is the set of parallel lines with some fixed direction. Choose a plane perpendicular to the lines and let $(u,v)$ be parameters on the plane. Then the vector field
\begin{equation}\label{e:rk0dv_free}
V=H(u,v)\frac{\partial}{\partial r},
\end{equation}
is divergence free for any real function $H$ where $\beta$ is given by equation (\ref{e:beta}) and $r$ is the distance to the plane.
\end{Prop}
\begin{proof}
If we fix the plane to be the $x^1x^2-$plane so that $x^1=u,x^2=v,x^3=r$ then for a fluid velocity 
\[
V=H(u,v,r)\frac{\partial}{\partial r},
\]
the divergence-free condition is simply
\[
\frac{\partial H}{\partial r}=0,
\]
as claimed.
\end{proof}

\vspace{0.1in}

\section{Rank Two Straight Streamlines}\label{s:rk2}

In this Section incompressible Euler flows with straight rank two streamlines are considered. As the associated oriented line congruence is of rank two throughout, it is given as the graph of a section $\eta=F(\xi,\bar{\xi})$, for some complex function $F$. The first order quantities $\sigma$, $\lambda$ and $\theta$  are defined by equations (\ref{e:spinco_rk2}).

For ease of exposition, the steady rank two fluid flow is solved first in Section \ref{s:rk2_s}. Non-steady rank two flows split into three cases: $\lambda^2-|\sigma|^2>0$, $\lambda^2-|\sigma|^2=0$ and $\lambda^2-|\sigma|^2<0$ and is solved in Section \ref{s:rk2_non_s}.

\subsection{Steady Rank Two Flow}\label{s:rk2_s}
\begin{Thm}
A rank two line congruence $\Sigma$ generates a steady solution of the incompressible Euler equations (\ref{e:euler1}) and (\ref{e:euler2}) iff $\Sigma$ is the set of normals to a round sphere and the fluid velocity vector and pressure are
\[
V=\frac{H_0}{r^2}\frac{\partial}{\partial r}  \qquad\qquad  p=p_0-\frac{H_0^2}{2r^4},
\]
for constants $H_0,p_0$, where $r$ is the distance to the centre of the sphere.
\end{Thm}
\begin{proof}
Impose the divergence-free condition on a vector field tangent to the line congruence by insisting from equation (\ref{e:rk2dv_free}) that 
\[
V^\xi=V^{\bar{\xi}}=0 \qquad\qquad V^r=\frac{H(\xi,\bar{\xi})}{(r+\theta)^2+\lambda^2-|\sigma|^2}.
\]

Writing the Euler equations in congruence coordinates $(\xi,\bar{\xi},r)$ leads to
\begin{equation}\label{e:euler_trans}
\frac{\partial V^\xi}{\partial t}=0 \iff \frac{\partial p}{\partial \xi}=-\frac{2\bar{\eta}}{(1+\xi\bar{\xi})^2}\frac{\partial p}{\partial r},
\end{equation}
and 
\[
\frac{\partial V^r}{\partial t}=0 \iff \frac{\partial p}{\partial r}=\frac{2(r+\theta)H^2}{((r+\theta)^2+\lambda^2-|\sigma|^2)^3}.
\]
This last can be integrated to yield
\[
    p=K(\xi,\bar{\xi})-\frac{H^2}{2((r+\theta)^2+\lambda^2-|\sigma|^2)^2}.
\]
Substituting this back in equation (\ref{e:euler_trans}) and comparing powers of $r$ one finds from the $r^6$ term that $K=p_0=constant$, while from the $r^2$ term that $H=H_0=constant$. The $r$ term is
\begin{equation}\label{e:cmroc}
\frac{\partial \theta}{\partial \xi}+\frac{2\bar{F}}{(1+\xi\bar{\xi})^2}=0.
\end{equation}
Differentiating this with respect to $\bar{\xi}$ and projecting onto the imaginary part yields
\[
\frac{\partial}{\partial \xi}\left(\frac{2{F}}{(1+\xi\bar{\xi})^2}\right)-\frac{\partial}{\partial \bar{\xi}}\left(\frac{2\bar{F}}{(1+\xi\bar{\xi})^2}\right)=\frac{2i\lambda}{(1+\xi\bar{\xi})^2}=0.
\]
where $\lambda$ is the twist of the line congruence defined in the second of equations (\ref{e:spinco_rk2}). Thus $\Sigma$ is Lagrangian and the lines are orthogonal to a 1-parameter family of surfaces in ${\mathbb R}^3$ \cite{gk08b}.  By equation (\ref{e:cmroc}) the orthogonal surface in ${\mathbb R}^3$ is of constant mean curvature. 

Finally the constant term is now 
\[
\frac{\partial (|\sigma|)}{\partial \xi}=0.
\]
The following identity holds between derivatives 
\[
(1+\xi\bar{\xi})^2\frac{\partial}{\partial \bar{\xi}}\left(\frac{\sigma}{(1+\xi\bar{\xi})^2} \right)=\frac{\partial (\theta+i\lambda)}{\partial \xi}+\frac{2\bar{F}}{(1+\xi\bar{\xi})^2},
\]
and so in our case, we know that 
\[
(1+\xi\bar{\xi})^2\frac{\partial}{\partial \bar{\xi}}\left(\frac{\sigma}{(1+\xi\bar{\xi})^2} \right)=0.
\]
If $\sigma\neq0$, and $\sigma=|\sigma|e^{i\phi}$ this reduces to 
\[
\frac{\partial}{\partial \bar{\xi}}\left(\frac{e^{i\phi}}{(1+\xi\bar{\xi})^2}\right).
\]
This is easily seen to be impossible, so we conclude that $\sigma=0$ and $\Sigma$ is the set of lines through a point. 

Moreover, the velocity and pressure of the fluid are as stated. In particular, if we translate the centre to the origin, then $\eta=0$ and so $\sigma=\lambda=\theta=0$ and
\[
V^r=\frac{H_0}{r^2}, \qquad \qquad p=p_0-\frac{H_0^2}{2r^4}. 
\]
Finally $r$ is the distance to the centre of the sphere.

\end{proof}

\vspace{0.1in}

\subsection{Non-Steady Rank Two Flow}\label{s:rk2_non_s}

We now drop the condition that the fluid flow be steady and find that the same result applies if the line congruence to which it is tangent remains fixed:
\vspace{0.1in}
\begin{Thm}
A fixed rank two line congruence $\Sigma$ generates a solution of the incompressible Euler equations iff $\Sigma$ is the set of normals to a round sphere, and the fluid velocity vector and pressure are those given in Example \ref{ex:sph}.
\end{Thm}
\begin{proof}
Writing the Euler equations in congruence coordinates $(\xi,\bar{\xi},r)$ we find that
\begin{equation}\label{e:euler_trans_f}
\frac{\partial p}{\partial \xi}=-\frac{2\bar{\eta}}{(1+\xi\bar{\xi})^2}\frac{\partial p}{\partial r},
\end{equation}
and 
\begin{equation}\label{e:euler_r_f}
\frac{1}{(r+\theta)^2+\lambda^2-|\sigma|^2}\frac{\partial H}{\partial t}=-\frac{\partial p}{\partial r}+\frac{2(r+\theta)H^2}{((r+\theta)^2+\lambda^2-|\sigma|^2)^3}.
\end{equation}
This last equation can be integrated explicitly in $r$ when the three cases (i) $\lambda^2-|\sigma|^2>0$,  (ii) $\lambda^2-|\sigma|^2>0$ and (iii) $\lambda^2-|\sigma|^2>0$ are treated separately. 

\vspace{0.1in}
\noindent{\bf Case (i)}:

Assume $\lambda^2-|\sigma|^2>0$. Equation (\ref{e:euler_r_f}) integrates to
\[
    p=K(\xi,\bar{\xi})-\frac{H^2}{2((r+\theta)^2+\lambda^2-|\sigma|^2)^2}-\frac{\dot{H}}{\sqrt{\lambda^2-|\sigma|^2}}\tan^{-1}\left(\frac{r+\theta}{\sqrt{\lambda^2-|\sigma|^2}}\right).
\]
where a dot represents differentiation with respect to time $t$. Substituting this in equation (\ref{e:euler_trans_f}) yields an equation of the form 
\[
\alpha(\xi,\bar{\xi},r)+\beta(\xi,\bar{\xi},r)\tan^{-1}\left(\frac{r+\theta}{\sqrt{\lambda^2-|\sigma|^2}}\right)=0,
\]
where the functions $\alpha$ and $\beta$ are 6th order polynomials in $r$. Thus we must have $\alpha=0$ and $\beta=0$ and each coefficient must vanish. 

The vanishing of the $r^6$ coefficient of $\alpha$ and $\beta$ yield $K=K_0=constant$ and
\[
\dot{H}=A(t)(\lambda^2-|\sigma|^2).
\]
Moving to the vanishing of the 5th and 4th orders of $\alpha$ we obtain $\lambda^2-|\sigma|^2=constant$, and 
\[
\frac{\partial \theta}{\partial \xi}+\frac{2\bar{F}}{(1+\xi\bar{\xi})^2}=0.
\]
As in the steady case, this implies that $\lambda=0$, which contradicts $\lambda^2>|\sigma|^2$.
Thus, there are no solutions for case (i).

\vspace{0.1in}
\noindent{\bf Case (ii)}:

Assume $\lambda^2-|\sigma|^2=0$. Equation (\ref{e:euler_r_f}) integrates to
\[
    p=K(\xi,\bar{\xi})-\frac{H^2}{2(r+\theta)^4}+\frac{\dot{H}}{r+\theta}.
\]
where a dot represents differentiation with respect to time $t$. Substituting this in equation (\ref{e:euler_trans_f}) yields an equation which is the vanishing of a 6th order polynomial in $r$. 

The vanishing of the $r^6$  and $r^4$ coefficients yield $K=p_0=constant$ and $\dot{H}=A(t)$, while the 3rd order implies
\[
\frac{\partial \theta}{\partial \xi}+\frac{2\bar{\eta}}{(1+\xi\bar{\xi})^2}=0.
\]
As before, we see that this implies $\lambda=0$ and since $\lambda^2-|\sigma|^2=0$, we have $\sigma=0$. Thus the line congruence is the lines through a single point and the fluid velocity and pressure are as stated.

\vspace{0.1in}
\noindent{\bf Case (iii)}:

Assume $\lambda^2-|\sigma|^2<0$. Equation (\ref{e:euler_r_f}) integrates to
\[
    p=K(\xi,\bar{\xi})-\frac{H^2}{2((r+\theta)^2+\lambda^2-|\sigma|^2)^2}-\frac{\dot{H}}{2\sqrt{|\sigma|^2-\lambda^2}}\ln\left(\frac{r+\theta-\sqrt{|\sigma|^2-\lambda^2}}{r+\theta+\sqrt{|\sigma|^2-\lambda^2}}\right).
\]
where a dot represents differentiation with respect to time $t$. Substituting this in equation (\ref{e:euler_trans_f}) yields an equation of the form 
\[
\alpha(\xi,\bar{\xi},r)+\beta(\xi,\bar{\xi},r)\ln\left(\frac{r+\theta-\sqrt{|\sigma|^2-\lambda^2}}{r+\theta+\sqrt{|\sigma|^2-\lambda^2}}\right)=0,
\]
where the functions $\alpha$ and $\beta$ are again 6th order polynomials in $r$, each coefficient of which must vanish. An identical calculation through descending powers of $r$ as case (i) yields the same result: $\sigma=\lambda=0$. This contradicts the assumption that $\lambda^2-|\sigma|^2<0$ and so there are no solutions in this case.

\end{proof}
\vspace{0.1in}
\section{Rank One Straight Streamlines}\label{s:rk1}
In this section incompressible Euler flows with straight rank one streamlines are considered. The steady flow case is proven first in Section \ref{s:rk1_s}, followed by the non-steady flow case in Section \ref{s:rk1_non_s}.

\subsection{Steady Rank One Flow}\label{s:rk1_s}
\begin{Thm}\label{t:rk1steady}
A rank one line congruence $\Sigma$ generates a steady solution of the incompressible Euler equations iff $\Sigma$ is the set of normals to a circular cylinder, and the fluid velocity vector and pressure are
\[
V=\frac{H_0}{r}\frac{\partial}{\partial r}  \qquad\qquad  p=p_0-\frac{H_0^2}{2r^2},
\]
for constants $H_0,p_0$, where $r$ is the distance to the centre of the cylinder.
\end{Thm}
\begin{proof}
As the associated oriented line congruence is of rank one throughout, it is given by a map $(u,v)\mapsto(\xi(u),\eta(u,v))$. The first order real quantity $\beta$ is defined by equation (\ref{e:beta}).

Consider then a rank one line congruence and divergence-free vector field given by equation (\ref{e:rk1dv_free}):
\[
V=\frac{H}{r+\beta}\frac{\partial}{\partial r},
\]
where $H$ is an arbitrary function  $(u,v)$ and $\beta$ is given by equation (\ref{e:beta}). Writing everything in terms of the congruence coordinates $(u,v,r)$ the first two Euler equations yield
\[
\frac{\partial p}{\partial u}=-\frac{2}{(1+\xi\bar{\xi})^2}\frac{\partial p}{\partial r}\left(\bar{\eta}\frac{d\xi}{d u}+\eta\frac{d\bar{\xi}}{d u}\right)
\qquad\qquad
\frac{\partial p}{\partial v}=0.
\]
The third Euler equation then gives
\[
\frac{\partial p}{\partial r}=\frac{H^2}{(r+\beta)^3},
\]
which integrates to
\begin{equation}\label{e:rk1press}
p=p_0-\frac{H^2}{2(r+\beta)^2}.
\end{equation}
Since we know that $p$ is independent of $v$ for all $r$, this means that $p_0,\beta$ and $H$ are functions only of $u$.

Substituting equation (\ref{e:rk1press}) into the first Euler equation we obtain an expression that is cubic in $r$. Thus each of the coefficients must vanish. In particular, from the cubic term $p_0=constant$ and the linear term $H=constant$, while the zeroth order term in $r$ now says that
\[
\frac{d \beta}{d u}=-\frac{2}{(1+\xi\bar{\xi})^2}\left(\bar{\eta}\frac{d\xi}{d u}+\eta\frac{d\bar{\xi}}{d u}\right).
\]
Only the complex function $\eta$ depends on $v$ and hence, after a suitable choice of this parameterization, the previous equation has solution
\begin{equation}\label{e:rk1eta}
\eta=\left(-\frac{(1+\xi\bar{\xi})^2}{4|\dot{\xi}|^2}\frac{d \beta}{d u}+iv\right)\frac{d\xi}{d u},
\end{equation}
where we introduce the dot for differentiation with respect to $u$.

Substituting for $\eta$ in the final Euler equation, the $r^2v$ term says that
\[
(1+\xi\bar{\xi})(\dot{\bar{\xi}}\ddot{\xi}-\dot{\xi}\ddot{\bar{\xi}})-2(\bar{\xi}\dot{\xi}-\xi\dot{\bar{\xi}})\dot{\xi}\dot{\bar{\xi}}=0.
\]
This is precisely the geodesic equation on the 2-sphere in holomorphic coordinates, and we conclude that the line congruence projects to a great circle on $S^2$.

By a rotation arrange that $\xi=u$ so that the great circle passes through the north pole ($\xi=0$) and aligns with the real axis. Now the vanishing of the $r^2$ term of the final Euler equation requires that
\[
(1+u^2)^2\ddot{\beta}+2u(1+u^2)\dot{\beta}+4\beta=0,
\]
which has general solution
\[
\beta=\frac{b_0+2b_1u-b_0u^2}{1+u^2},
\]
for constants $b_0,b_1$.

Substituting this into equation (\ref{e:rk1eta}) we find the line congruence must be
\[
\xi=u
\qquad\qquad
\eta=-{\scriptstyle{\frac{1}{2}}}(b_1-2b_0u-b_1u^2)+iv.
\]
Finally, a translation in the $x^1x^3-$plane sets $b_0$ and $b_1$ to zero and we get the line congruence 
\[
\xi=u
\qquad\qquad
\eta=iv.
\]
which consists of all of the horizontal lines that intersect the $x^3-$axis. These are the lines normal to a circular cylinder, and the velocity vector and pressure are found to be as stated in the Theorem where now  $r$ is the distance to the centre of the cylinder.
\end{proof}

\subsection{Non-Steady Rank One Flow}\label{s:rk1_non_s}

We now drop the condition that the fluid flow be steady:
\vspace{0.1in}
\begin{Thm}
A fixed rank one line congruence $\Sigma$ generates a solution of the incompressible Euler equations iff $\Sigma$ is the set of normals to a circular cylinder, and the fluid velocity vector and pressure are those given in Example \ref{ex:cyl}.
\end{Thm}
\begin{proof}
Writing the Euler equations in congruence coordinates $(u,v,r)$ we find that
\begin{equation}\label{e:rk1trans}
\frac{\partial p}{\partial u}=-\frac{2}{(1+\xi\bar{\xi})^2}\frac{\partial p}{\partial r}\left(\bar{\eta}\frac{d\xi}{d u}+\eta\frac{d\bar{\xi}}{d u}\right)
\qquad\qquad
\frac{\partial p}{\partial v}=0.
\end{equation}
The third Euler equation then gives
\[
\frac{1}{r+\beta}\frac{\partial H}{\partial t}=-\frac{\partial p}{\partial r}+\frac{H^2}{(r+\beta)^3},
\]
which integrates to
\[
p=p_0-\frac{H^2}{2(r+\beta)^2}-\dot{H}\ln|r+\beta|.
\]
for function $p_0(u,v,t)$. In fact, since $p$ must be independent of $v$ (for all $r$) we conclude that that $p_0$ and $H$ can only depend on $u$.

Substituting this back into the first of equations (\ref{e:rk1trans}) and proceeding by the order of $r$ in a manner similar to the proof of Theorem \ref{t:rk1steady}, the result follows.

\end{proof}

\vspace{0.1in}
\section{Rank Zero Straight Streamlines}\label{s:rk0}
In this section incompressible Euler flows with straight rank zero streamlines are considered.

\begin{Thm}\label{t:rk0}
A rank zero oriented line congruence $\Sigma$ generates a solution of the incompressible Euler equations iff $\Sigma$ is the set of normals to a flat plane, and the fluid velocity vector and pressure are those given in Example \ref{ex:pln}.
\end{Thm}
\begin{proof}
A rank zero oriented line congruence is a 2-parameter family of parallel lines and thus forms the set of normals to a flat plane. By a rotation we can set the direction of the oriented lines to be the positive $x^3-$direction and so $\xi=0$. Parameterize the plane by $(x^1,x^2)=(u,v)$ and so $\eta=(u+iv)/2$. 

We know from equation (\ref{e:rk0dv_free}) that, as a time varying divergence-free vector field tangent to the oriented line congruence, the fluid velocity is
\[
V=F(u,v,t)\frac{\partial}{\partial r},
\]
for some function $F$, where $r=x^3$.

Two of the Euler equations now say that $p$ is independent of $u$ and $v$, while the third says that
\[
\frac{\partial F}{\partial t}=-\frac{\partial p}{\partial r}.
\]
We conclude that $F=H(t)+K(u,v)$ for some functions $H$ and $K$ and 
\[
p=p_0-\frac{\partial H}{\partial t}r,
\]
as claimed.
\end{proof}

\end{document}